\documentclass[12pt]{article}
\usepackage{hyperref}
\usepackage[utf8]{inputenc}

\usepackage{fancyhdr}

\usepackage{bm}
\usepackage[margin=1in]{geometry}

\usepackage{mathtools,amssymb,amsthm,esint}
\usepackage{yhmath}
\usepackage{multicol}

\usepackage{enumitem}

\usepackage{cancel}

\newtheorem{theorem}{Theorem}[section]

\newtheorem{lemma}[theorem]{Lemma}
\newtheorem{conjec}{Conjecture}[section]
\theoremstyle{definition}
\newtheorem{define}{Definition}[section]
\usepackage{tikz-cd}
\usepackage{centernot}

\newcommand{\verts}[1]{\left\lvert #1 \right\rvert}

\newcommand{\abs}{\verts}

\usepackage{array}
\newcolumntype{L}{>{$}l<{$}}
\newcolumntype{C}{>{$}c<{$}}
\newcolumntype{R}{>{$}r<{$}}

\newcommand{\setp}{\mathbb P}

\newcommand{\pr}{\setp}

\newcommand{\Mod}[1]{\ (\text{mod}\ #1)}

\title{Nonuniform Distributions of Residues of Prime Sequences in Prime Moduli\vspace{-1ex}}
\author{David Wu}
\date{\vspace{-10ex}}
\usepackage{setspace}
\usepackage{titling}
\usepackage{caption}

\renewcommand{\vec}{\mathbf}

\newcommand{\li}{\mathrm{li}}
\newcommand{\tset}{\mathcal{T}}
\newcommand{\aset}{\mathcal{A}}
\usepackage{physics}
\usepackage{amsfonts}
\usepackage{float}
\numberwithin{equation}{section}

\usepackage{standalone}

\begin{document}
\maketitle

\begin{abstract}
    For positive integers $q$, Dirichlet's theorem states that there are infinitely many primes in each reduced residue class modulo $q$. A stronger form of the theorem states that the primes are equidistributed among the $\varphi(q)$ reduced residue classes modulo $q$. This paper considers patterns of sequences of consecutive primes $(p_n, p_{n+1}, \ldots, p_{n+k})$ modulo $q$. Numerical evidence suggests a preference for certain prime patterns. For example, computed frequencies of the pattern $(a,a)$ modulo $q$ up to $x$ are much less than the expected frequency $\pi(x)/\varphi(q)^2$. We begin to rigorously connect the Hardy-Littlewood prime $k$-tuple conjecture to a conjectured asymptotic formula for the frequencies of prime patterns modulo $q$. 
\end{abstract}

\section{Introduction}
Analytic number theory uses real and complex analysis techniques to prove properties about the integers. It turns out that many properties of prime numbers are encoded in the properties of special functions. For example, the behavior of the zeros of the Riemann zeta function strengthens a famous asymptotic formula known as the Prime Number Theorem (PNT) \cite{apostol}. The Riemann Hypothesis (RH), one of the most well-known open problems in number theory, conjectures that all nontrivial zeros of the Riemann zeta function have real part $\frac{1}{2}$; RH would imply a stronger form of PNT. 

The Riemann zeta function is only one of a more general class of functions, the Dirichlet $L$-functions. Peter Dirichlet \cite{dirichlet_2013} used these $L$-functions to prove that arithmetic progressions with coprime first term and common difference contain an infinite number of primes: this is Dirichlet's theorem. Dirichlet's use of $L$-functions invoked the realm of analysis to prove statements about integers, thus beginning the study of analytic number theory.

The $\varphi(d)$ classes of residues modulo $d$ coprime to $d$ are referred to as the reduced residue classes, where $\varphi(n)$ is Euler's totient function. For example, the set of residues congruent to $1$ modulo $4$ is a reduced residue class. Applying Dirichlet's theorem to the arithmetic progression with first term $1$ and common difference $4$ shows that there are infinitely many primes in the reduced residue class $1$ modulo $4$. A natural followup question asks how prime sequences are distributed among the reduced residue classes modulo $d$. 

Before we discuss the distribution of primes among reduced residue classes, we introduce a few standard definitions. Let $\pi(x)$ be the usual prime counting function, i.e. the number of primes less than or equal to $x$. Furthermore, let $p(x) \sim q(x)$ denote asymptotic equivalence, i.e. $\lim\limits_{x \to \infty} \frac{p(x)}{q(x)} = 1$. We also make extensive use of big $O$ notation. We say $f(x) = O(g(x))$ if there exists some absolute constant $C$ such that $\abs{f(x)} \le C\abs{g(x)}$ for sufficiently large $x$. The similar notation $f(x) = O_n(g(x))$ means the constant $C$ in the definition of $O(g(x))$ depends on $n$. 

A key idea in analytic number theory is to compare a discrete function such as $\pi(x)$ to a continuous function such as the logarithmic integral $\li(x) = \int_{2}^x \frac{dt}{\log t}$. The famous Prime Number Theorem states that $\pi(x) \sim \li(x)$, and Schoenfeld \cite{schoenfeld1976} showed that RH implies that $\abs{\pi(x)-\li(x)} < \frac{\sqrt{x}\log x}{8\pi}$, 
for $x \ge 2657$. 

We introduce notation analogous to $\pi(x)$ for the purposes of this discussion following Lemke Oliver and Soundararajan's notation \cite{oliver2016}. Let $p_n$ refer to the $n$th prime when the primes are listed in increasing order, the pattern $\vec{a} = (a_1, a_2, \ldots, a_k)$ be a vector of length $k$, and $q \ge 3$ be a positive integer. Define \[\pi(x; q, \vec{a}) = \#\{p_n \le x : p_{n+i-1} \equiv a_i \Mod{q} \text{ for } 1 \le i \le k\}.\] This notation counts the number of consecutive prime sequences that follow the pattern $\vec{a}$ modulo $q$. Using this notation, the PNT for arithmetic progressions applied to the simple case where $\vec{a} = (a)$ yields \begin{equation}\label{introeq}
\pi(x; q, a) \sim \frac{\li(x)}{\varphi(q)}.
\end{equation} 

Although \eqref{introeq} shows that primes are roughly equidistributed among the reduced residue classes modulo $q$, Chebyshev \cite{cheb} observed that there are almost always more primes of the form $4k+3$ than of the form $4k+1$; this bias was explained by Rubinstein and Sarnak \cite{rubinstein1994} to arise from the error term of $O(x^{1/2+\epsilon})$ in PNT when assuming RH. Chebyshev's bias is one of the first mentions of nonuniform behavior of the primes when reduced modulo $q$. 


Larger biases manifest when the length of $\vec{a}$ is greater than or equal to $2$ that cannot be solely attributed to error terms of size $O(x^{1/2+\epsilon})$. In \cite{oliver2016} the frequencies of consecutive prime pairs modulo $10$ are tabulated, and it was observed that $\pi(10^8; 10, (1,1)) \approx 4.62 \times 10^6$ and $\pi(10^8; 10, (9, 1)) \approx 7.99 \times 10^6$, both of which are very different than the expected frequency of $10^8/\varphi(10)^2 = 6.25 \times 10^6$ predicted by naively generalizing \eqref{introeq} by replacing $\varphi(q)$ with $\varphi(q)^2$.

While it is known that primes are roughly equidistributed among reduced residue classes according to \eqref{introeq}, it is not known whether for arbitrary $\vec{a}$ with the length of $\vec{a}$ at least $2$, the modified prime counting function $\pi(x; q, \vec{a})$ tends to infinity as $x$ tends to infinity. Shiu \cite{shiu2000} proved that $\pi(x; q, (a, a, \ldots, a))$ tends to infinity as $x$ tends to infinity, and Maynard \cite{maynard2016} strengthened this to $\pi(x; q, (a, a, \ldots, a)) > C\pi(x)$ for some constant $C$ and sufficiently large $x$. 

We can explain the preferences for certain prime patterns by appealing to conjectural statements similar in nature to the PNT. For example, the Hardy-Littlewood prime $k$-tuple conjecture gives us the density of specific tuples such as twin primes $(p, p+2)$ and twin sexy primes $(p, p+2, p+6)$ in a form analogous to that of the PNT. By appropriately combining specific cases of the Hardy-Littlewood prime $k$-tuple conjecture, we obtain conjectures about the density of the patterns modulo $q$. 

As an example of how specific prime tuples relate to prime patterns, consider $q=3$ and the pattern $\vec{a} = (1,1)$. Then we are restricting our consideration to consecutive primes $(p_1, p_2)$ where $p_1 \equiv p_2 \equiv 1 \pmod{3}$. For $m \equiv 1 \pmod{3}$, these patterns include specific tuples of the form $(m, m+6), (m, m+12), (m, m+18),$ and so on. The densities of these specific tuples can be analyzed with the Hardy-Littlewood prime $k$-tuple conjectures. In this manner, we obtain conjectures that partially account for the observed preferences for certain patterns modulo $q$. 

Lemke Oliver and Soundararajan \cite{oliver2016} provide a conjectural explanation for the biases for certain prime patterns. However, their heuristic argument omits lower order terms that cause their conjectured form to not be in agreement with the data at smaller values of $x$. We expand the conjecture to include further terms and begin rigorously connecting the Hardy-Littlewood prime $k$-tuple conjecture and the main conjecture in \cite{oliver2016}.

In Section \ref{sec:def}, we lay out the definitions and notation important to our discussion. In Section \ref{sec:conjec}, we determine the lower order terms by tightening the asymptotics in the heuristic in \cite{oliver2016}. In Section \ref{app:a}, we prove the lemmas necessary for the main proof. In Appendices \ref{sec:numeric} and \ref{app:c}, we account for discarded terms in the asymptotic formula for the conjectured behavior to extend the form of an integral to more closely fit the actual behavior of prime patterns and extend our data gathering capabilities by 8 orders of magnitude. We identify a plausible lower order term for the conjectured formula. 

\section{Preliminaries}\label{sec:def}
We begin with the statement of the Hardy-Littlewood prime $k$-tuple conjecture. Heuristically, the conjecture generalizes the PNT by assuming the probability of an integer $n$ being prime as roughly $\frac{1}{\log n}$. While the integrand is derived by assuming primality is independent, the constant in front of the integral corrects for this assumption. 

\begin{conjec}[Hardy-Littlewood prime $k$-tuple conjecture]
Let $\mathcal{H}$ be a finite set of nonnegative integers and $\pi(x, \mathcal{H})$ denote the number of integers $n \le x$ such that $n+h$ is a prime for all $h$ in $\mathcal{H}$. Furthermore, let $\nu_p(\mathcal{H})$ denote the number of residue classes occupied by the members of $\mathcal{H}$ modulo $p$. Then we have that \[\pi(x, \mathcal{H}) = \mathfrak{S}(\mathcal{H})\int_2^x \frac{dt}{(\log t)^{\abs{\mathcal{H}}}} + O(x^{1/2+\epsilon}),\] where the singular series is defined as \[\mathfrak{S}(\mathcal{H}) = \prod_{p \text{ prime}} \frac{1 - \frac{\nu_p(\mathcal{H})}{p}}{(1 - \frac{1}{p})^{\abs{\mathcal{H}}}}.\]
\end{conjec}

The singular series is modified in \cite{montgomery2004} to an inclusion-exclusion form \[\mathfrak{S}_{0}(\mathcal{H}) = \sum\limits_{\tset \subset \mathcal{H}} (-1)^{\abs{\mathcal{H}\setminus\tset}}\mathfrak{S}(\tset).\] In \cite{oliver2016}, Lemke Oliver and Soundararajan modify the singular series to range over primes $p$ not dividing $q$ to account for the prime patterns modulo $q$ as follows. 

\begin{define}\label{def:1}
The modified singular series $\mathfrak{S}_q(\mathcal{H})$ is defined to be  \[\mathfrak{S}_q(\mathcal{H}) = \prod_{p \nmid q} \frac{1 - \frac{\nu_p(\mathcal{H})}{p}}{(1 - \frac{1}{p})^{\abs{\mathcal{H}}}}.\] 
\end{define}

Lemke Oliver and Soundararajan \cite{oliver2016} introduce the same inclusion-exclusion form $\mathfrak{S}_{q,0}$ involving alternating sums of  $\mathfrak{S}_{q}$ is defined to introduce cancellations that lead to Conjecture \ref{conjec:1}.  

Let $q \ge 3$ be a positive integer and $a$ and $b$ be reduced residue classes modulo $q$. Set $h \equiv b-a \pmod{q}$. Also, let $p_n$ be the $n$th prime. We are specifically interested in the case where $p_n \equiv a \pmod{q}$ and $p_{n+1} = p_n+h$; this guarantees $p_{n+1} \equiv b \pmod{q}$.  Let $1_\mathcal{P}(x)$ be the prime indicator function, defined to be $1$ if $x$ is prime and $0$ otherwise; Lemke Oliver and Soundararajan \cite{oliver2016} start with the statement that \begin{equation}\label{eq:25}
\pi(x;q,a,b) = \sum_{\substack{n \le x \\ n \equiv a \pmod{q}}}1_\mathcal{P}(n)1_\mathcal{P}(n+h) \prod_{\substack{0 < t < h \\ (t+a, q) = 1}}(1-1_\mathcal{P}(n+t)).
\end{equation} Following a series of manipulations and using a conjecture similar to the Hardy-Littlewood prime $k$-tuple conjecture, 
they conjecture the following asymptotic for $\pi(x; q, (a, b))$ (see \cite[E4449--E4450]{oliver2016} for more details).

\begin{conjec}[Lemke Oliver \& Soundararajan  \cite{oliver2016}]\label{conjec:1}
Let  \[\alpha(y) = 1 - \frac{q}{\varphi(q)\log y} \hspace{.25 cm}\text{ and } \hspace{.25cm} \epsilon_q(a,b) = \#\{0<t<h: (t+a,q) =1\} - \frac{\varphi(q)}{q}h.\] Then 
\[\pi(x; q, (a, b)) \sim \frac{1}{q}\int_{2}^{x} \alpha(y)^{\epsilon_q(a,b)}\qty(\frac{q}{\varphi(q)\alpha(y)\log y})^2 \mathcal{D}(a,b;y) dy,\]
where $\mathcal{D}(a,b;y)$ is defined to be \[\sum_{\substack{h >0 \\ h \equiv b-a \pmod{q}}} \sum_{\aset \subset \{0, h\}} \sum_{\substack{\tset \subset [1, h-1] \\ (t+a, q)=1\,\,\forall t \in \tset}} (-1)^{\abs{\tset}} \mathfrak{S}_{q,0}(\aset \cup \tset) \qty(\frac{q}{\varphi(q)\alpha(y)\log y})^{\abs{\tset}} \alpha(y)^{h\varphi(q)/q}.\]

\end{conjec}

We analyze the growth of $\mathcal{D}(a, b;y)$. For readability purposes, define $\log_k x$ to be \\$\underbrace{\log \log \ldots \log}_{k \text{ logs}} x$, where $\log x$ is the natural logarithm. 

\section{A Closer Analysis of the Conjecture}\label{sec:conjec}
We provide more precise asymptotics for $\pi(x; q, (a, b))$ as in Conjecture \ref{conjec:1}. Because $q=2$ is trivial, we only consider the case where $q$ is an odd prime. However, the results readily generalize to composite $q$. In particular, we are interested in $\mathcal{D}(a,b;y)$, which is equal to
\begin{equation}\label{eq:1}
\sum_{\substack{h >0 \\ h \equiv b-a \pmod{q}}} \sum_{\aset \subset \{0, h\}} \sum_{\substack{\tset \subset [1, h-1] \\ (t+a, q)=1\,\,\forall t \in \tset}} (-1)^{\abs{\tset}} \mathfrak{S}_{q,0}(\aset \cup \tset) \qty(\frac{q}{\varphi(q)\alpha(y)\log y})^{\abs{\tset}} \alpha(y)^{h\varphi(q)/q},
\end{equation}
in accordance with \cite{oliver2016}. Lemke Oliver and Soundararajan heuristically argue that the relevant terms in \eqref{eq:1} are those where $\aset = \tset = \varnothing$ and $\abs{\aset} + \abs{\tset} = 2$. 

We convert \eqref{eq:1} into a form more friendly to partitioning by the size of $\tset$. Define for convenience \[z = z(q, y) = \frac{q}{\varphi(q)\alpha(y)\log y}\] and \[g = g(q,y) = \alpha(y)^{\varphi(q)/q}.\] We rewrite the innermost sum of \eqref{eq:1} as a sum over $\ell$ element subsets of $[1, h-1]$ where $\ell$ ranges from $0$ to $h-1$ to obtain \begin{equation}\label{eq:2}
    \mathcal{D}(a,b;y) = 
    \sum_{\substack{h >0 \\ h \equiv b-a \pmod{q}}} g^h \sum_{\aset\subset\{0,h\}}\sum_{\ell=0}^{h-1} (-z)^{\ell} \sum_{\substack{\tset \subset [1, h-1] \\ (t+a, q)=1\,\,\forall t \in \tset\\ \abs{\tset} = \ell} } \mathfrak{S}_{q,0}(\aset \cup \tset).
\end{equation}

Evaluating \eqref{eq:2} is difficult because the terms are unwieldy when $h$ is large. However, recalling the role of $h$ in \eqref{eq:25}, we see that large $h$ correspond to large prime gaps. Lemma \ref{lemma:3} constrains the behavior of large prime gaps and hence of \eqref{eq:2} when $h$ is large. 

Let $c$ be a sufficiently large positive integer depending on $n$ and define $M = c\log_2 y$. We split the outermost sum over $h$ in \eqref{eq:2} into two regions: One with $0 < h \le M\log y$ and one with $h > M\log y$. The sum where $h > M\log y$ counts contributions where $g_n > M\log y$. However, this portion of the sum can only contribute if its terms exist at all, therefore, the sum where $h > M\log y$ is bounded above by the probability that $g_n>M\log y$. Hence, by Lemma \ref{lemma:3}, the sum where $h > M\log y$ is bounded above by $\frac{1}{\log^c y}$. Thus, by controlling $c$, we can discard the portion of the sum where $h > M\log y$. For the remainder of this paper, we consider $h \le M\log y$. 

For $n = 0, 1, 2$, Lemke Oliver and Soundararajan define $\mathcal{D}_n(a,b;y)$ to be the terms obtained from \eqref{eq:1} where $\abs{\tset} = n$ and $\aset = \tset = \varnothing$ or $\abs{\aset} + \abs{\tset} = 2$. However, note that $\mathcal{D}_n(a,b;y)$ is precisely the term obtained by isolating the $\ell = n$ term in \eqref{eq:2}. Starting the sum over $\ell$ in \eqref{eq:2} at $\ell = n$ rather than $\ell = 0$ is the first step towards investigating $\mathcal{D}_n(a,b;y)$. Define  $\mathcal{D}_{\ge n}(a, b; y) = \sum\limits_{i \ge n}^{M\log y} \mathcal{D}_i(a,b;y)$. Written explicitly, the terms of \eqref{eq:2} we are interested in are

\begin{equation}\label{eq:4}
    \mathcal{D}_{\ge n}(a, b;y) = \sum_{\substack{0 < h \le M\log y \\ h \equiv b-a \pmod{q}}} g^h \sum_{\aset \subset \{0, h\}}\sum_{\ell=n}^{h-1} (-z)^{\ell} \sum_{\substack{\tset \subset [1, h-1] \\ (t+a, q)=1\,\,\forall t \in \tset\\ \abs{\tset} = \ell} } \mathfrak{S}_{q,0}(\aset \cup \tset).
\end{equation}

Furthermore, define \begin{align*}
     & A_{h, \ell} = \displaystyle\sum_{\substack{\tset \subset [1, h-1] \\ (t+a, q)=1\\ \abs{\tset} = \ell}} \mathfrak{S}_{q,0}(\tset),      &&B_{h, \ell} = \displaystyle\sum_{\substack{\tset \subset [1, h-1] \\ (t+a, q)=1\\ \abs{\tset} = \ell} } \mathfrak{S}_{q,0}(\{0\} \cup \tset), \\
     &C_{h, \ell} = \displaystyle\sum_{\substack{\tset \subset [1, h-1] \\ (t+a, q)=1\\ \abs{\tset} = \ell} } \mathfrak{S}_{q,0}(\{h\} \cup \tset), 
     &&D_{h, \ell} = \displaystyle\sum_{\substack{\tset \subset [1, h-1] \\ (t+a, q)=1\\ \abs{\tset} = \ell} } \mathfrak{S}_{q,0}(\{0,h\} \cup \tset).
\end{align*}

We partition the summation in \eqref{eq:4} into four terms $S_{\varnothing}$, $S_{\{0\}}$, $S_{\{h\}}$, and $S_{\{0, h\}}$, based on $\aset$.
For example, 
\begin{equation}\label{eq:42}S_{\varnothing} = \sum_{\substack{0 < h < M\log y \\ h \equiv b-a \pmod{q}}} g^h \sum_{\ell=n}^{h-1} (-z)^{\ell} A_{h, \ell},\end{equation} with $S_{\{0\}}$, $S_{\{h\}}$, and $S_{\{0, h\}}$ defined analogously with sums over $B_{h, \ell}$, $C_{h,\ell}$, and $D_{h,\ell}$, respectively.

In order to handle $A_{h, \ell}$, $B_{h, \ell}$, $C_{h, \ell}$, and $D_{h, \ell}$, we modify the following result of Montgomery and Soundararajan \cite{montgomery2004}, which states the average order of $\mathfrak{S}_0$. They show that
\begin{equation}\label{eq:40}
\sum_{\substack{\tset \subset [1, h] \\ \abs{\tset} = \ell}} \mathfrak{S}_0(\tset) = \frac{\mu_{\ell}}{\ell!}(-h\log h+Ah)^{\ell/2} + O(h^{\ell/2-1/7\ell+\epsilon}),
\end{equation}
where $\mu_{\ell}$ is the $\ell^{\text{ th}}$ moment of the standard normal distribution and $A$ is an absolute constant between $-1$ and $0$. We expect that $\sum \mathfrak{S}_{q,0}(\tset)$ has a similar growth rate, up to minor corrections such as the exact value of $A$ and leading factors depending on $q$. Moreover, these arguments used to justify Theorem \ref{thm:1} are expected to be robust against these modifications. 

We prove the following theorem concerning the growth rates of $S_{\varnothing}$, $S_{\{0\}}$, $S_{\{h\}}$, and $S_{\{0, h\}}$, which proves a weaker version of the claim in \cite{oliver2016} that $\mathcal{D}_n(a, b; y)$ is $O_n\qty(\frac{(\log_2y)^{n/2}}{(\log y)^{n/2-1}})$.

\begin{theorem}\label{thm:1}
Assuming that \eqref{eq:40} holds in a similar form for $\mathfrak{S}_{q,0}$, we have that $S_{\varnothing}$,\, $S_{\{0\}}\log y$,\, $S_{\{h\}}\log y$,\, and $S_{\{0, h\}}(\log y)^2$ are all
\[O_n\qty( \frac{(\log_2 y)^{n}}{(\log y)^{n/2-1}}).\] In particular, $\mathcal{D}_{n}(a,b;y)$ and $\mathcal{D}_{\ge n}(a, b;y)$ are both $O_n\qty( \frac{(\log_2 y)^{n}}{(\log y)^{n/2-1}})$, allowing us to truncate \\$\mathcal{D}(a,b;y)$ at specific values of $n$ and control the error terms in Conjecture \ref{conjec:1}. 
\end{theorem}

 We defer the proofs of Lemmas \ref{lemma:1}-\ref{lemma:4} used in the proof of Theorem \ref{thm:1} to Section \ref{app:a}.
 
\begin{proof}

We begin by evaluating $S_{\varnothing}$ according to \eqref{eq:42}. We are interested in the case where $q$ is prime, and thus $\varphi(q) = q-1$ and $\alpha(y) = 1 - \frac{q}{(q-1)\log y}$. Because $q$ is an odd prime, $\frac{\varphi(q)}{q} \ge \frac{2}{3}$. Thus, \[1 - \frac{3}{2\log y} \le \alpha(y) < 1 - \frac{1}{\log y}.\] For sufficiently large $y$, the definition of $z$ gives \[z = \frac{q}{\varphi(q)\alpha(y)\log y} < \frac{3}{2(1-\frac{3}{2\log y})\log y} < \frac{3}{\frac{3}{2}\log y} = \frac{2}{\log y }.\]  

Appealing to our conjectured form for $\sum \mathfrak{S}_{q,0}$ according to \eqref{eq:40}, we replace $A_{h, \ell}$ in \eqref{eq:42} with  $\frac{\mu_{\ell}}{\ell!}(-h\log h + Ah)^{\ell/2} + O(h^{\ell/2 - 1/7\ell + \epsilon})$. Note that $\mu_\ell =0$ when $\ell$ is odd, so we analyze the sum based on the parity of $\ell$. 

\noindent\emph{Case 1:} $\ell$ is even. For convenience, define $m = \ell/2$. We split the single sum over $h$ into a sum over $j$, $k$ and $h$ and swap the order of summation so that $\eqref{eq:42}$ is less than \begin{equation}\label{eq:5}\sum_{m = \frac{n}{2}}^{M\log y-1} \sum_{j = 0}^{\frac{M }{\log_3 y}} \sum_{k = j\log_3 y}^{(j+1)\log_3 y-1}
    \sum_{\substack{h = k\log y+1\\ h \equiv b-a \pmod{q}}}^{(k+1)\log y} g^h \qty(\frac{2}{\log y})^{2m}\qty(\frac{\mu_{2m}}{(2m)!}(-h\log h + Ah)^{m} ). 
\end{equation}

We bound \eqref{eq:5} above by a series of substitutions. Define $B = -A > 0$ and take the absolute value of the terms of \eqref{eq:5}. Lemma \ref{lemma:1} implies $(h\log h + Bh)^{m}$ has an upper bound of  $2^m[(h\log h)^m + (Bh)^m]$, where we include the extra factor of $2$ for convenience. Because $g = \qty(1 - \frac{q}{\varphi(q)\log y})^{\varphi(q)/q}$ and $\frac{\varphi(q)}{q} \ge \frac{2}{3}$, We have \[g < \qty(1 - \frac{3}{2\log y})^{2/3} < e^{-2/(3\log y)}.\] Thus, $g^h$ has an upper bound of $e^{-2j\log_3 y\log y/(3\log y)} = (\log_2 y)^{-2j/3}$. We then maximize all instances of $h$ by replacing $h$ with $h_{\max} = (k+1)\log y$ and remove the sum over $h$ by multiplying the summand by $\log y$. Finally, note that $\mu_{2m} = (2m-1)!!$, so $\frac{\mu_{2m}}{(2m)!} = \frac{1}{2^m m!}$. These substitutions yield  \begin{equation}\label{eq:6}
\sum_{m = \frac{n}{2}}^{M\log y-1} \sum_{j = 0}^{\frac{M}{\log_3 y}} \sum_{k = j\log_3 y}^{(j+1)\log_3 y-1} \frac{(\log y)^{1-2m}}{(\log_2 y)^{2j/3}}\qty(\frac{2^{2m}}{2^m m!}(2^m[(h_{\max}\log h_{\max})^m + (Bh_{\max})^{m}])). 
\end{equation}

Applying Lemma \ref{lemma:1} to $(\log h_{\max})^m = (\log(k+1) + \log_2 y)^m$ implies  \begin{equation}\label{eq:7}(h_{\max}\log h_{\max})^m \le \qty(2(k+1)\log y)^m[(\log(k+1))^m + (\log_2 y)^m].\end{equation} Substituting \eqref{eq:7} into \eqref{eq:6}, distributing the factor of $(2/\log y)^{2m}$, and cancelling the factor of $2^m$ yields \begin{multline*}
\sum_{m = \frac{n}{2}}^{M\log y-1} \sum_{j = 0}^{\frac{M}{\log_3 y}} \sum_{k = j\log_3 y}^{(j+1)\log_3 y-1} \frac{\log y}{(\log_2 y)^{2j/3}}\Bigg(\frac{1}{m!}\bigg[\qty(\frac{8(k+1)\log(k+1)}{\log y})^m \\+ \qty(\frac{8(k+1)\log_2 y}{\log y})^m\bigg]\Bigg).
\end{multline*}

Again, we maximize $k$ and remove the sum over $k$ by multiplying by $\log_3 y$, leaving
\begin{multline}\label{eq:13}
\sum_{m = \frac{n}{2}}^{M\log y-1} \sum_{j = 0}^{\frac{M}{\log_3 y}}  \frac{\log y\log_3 y}{(\log_2 y)^{2j/3}}\Bigg(\frac{1}{m!}\bigg[\qty(\frac{8((j+1)\log_3 y)(\log(j+1) + \log_4 y))}{\log y})^m  \\ +\qty(\frac{8(j+1)\log_2 y\log_3 y}{\log y})^m\bigg]\Bigg).
\end{multline}

We split the sum in \eqref{eq:13} up into four cases based on the value of $j$. 

\noindent\emph{Case 1A:} $j = 0$. 
When $j=0$, the sum in \eqref{eq:13} becomes
\begin{equation}\label{eq:33}
\log y\log_3 y\sum_{m=\frac{n}{2}}^{M\log y-1} \frac{1}{m!}\qty[\qty(\frac{8\log_3y\log_4y}{\log y})^m + \qty(\frac{8\log_2y\log_3y}{\log y})^m].
\end{equation} 

Note that \eqref{eq:33} is a truncated Taylor polynomial of $e^x$. We show that \eqref{eq:33} is $O(f(n))$, where $f(n)$ is the first term of the truncated Taylor polynomial. With this in mind, because the summation in \eqref{eq:33} is a truncated series of positive terms, it is less than the value of the complete Taylor series 
$e^{8\log_3 y\log_4 y/\log y} + e^{8\log_2 y\log_3y/\log y}$. 

Simplifying and noting that $(\log_a y)^b$ is $O(\log y)$ for any $a \ge 2$ and $b \ge 0$, Lemma \ref{lemma:2}, whose statement and proof can be found in Appendix \ref{app:a}, implies that the expression is $O(1)$. Because the Taylor series is $O(1)$, the growth rate of \eqref{eq:33} for varying $n$ is determined by the first term. Hence, \eqref{eq:33} is \begin{equation}\label{eq:34}
O_n\qty(\frac{(\log_2y)^{n/2}(\log_3y)^{n/2+1}}{(\log y)^{n/2-1}}). 
\end{equation}

\noindent\emph{Case 1B:} $j = 1$. 
Analyzing the $j=1$ term follows similar logic; the asymptotic we obtain is also \begin{equation}\label{eq:35}
O_n\qty(\frac{(\log_2y)^{n/2}(\log_3y)^{n/2+1}}{(\log y)^{n/2-1}}).
\end{equation}

\noindent\emph{Case 1C:} $2 \le j < \frac{3\log_2y}{2\log_3 y}$.
Because $j \ge 2$, we know \[(\log_2 y)^{-2j/3} < (\log_2 y)^{-4/3} < \frac{1}{\log_2 y}.\] We also know that $j+1 \le \frac{3\log_2 y}{2\log_3 y}$. Maximizing $(\log_2 y)^{-2j/3}$ and $j+1$, removing the summation by multiplying by $\frac{3\log_2 y}{2\log_3 y}$, and cancelling $\log(\frac{\log_2 y}{\log_3 y})$ with $\log_4 y$ yields \begin{equation}\label{eq:14}
\frac{3\log y}{2} \sum_{m=\frac{n}{2}}^{M\log y-1}\frac{1}{m!} \qty[ \qty(\frac{12\log_2 y\log_3 y}{\log y})^m + \qty(\frac{12(\log_2 y)^2}{\log y})^m].
\end{equation}

As before, the sum in \eqref{eq:14} is a truncated Taylor series that is $O(1)$. Hence, \eqref{eq:14} is 
\begin{equation}\label{eq:16}
O_n\qty(\frac{(\log_2y)^{n}}{(\log y)^{n/2-1}}).
\end{equation}

\noindent\emph{Case 1D:} $\frac{3\log_2 y}{2\log_3 y} \le j \le \frac{M}{\log_3 y}$. 
When $j > \frac{3\log_2 y}{2\log_3 y}$, the factor $(\log_2 y)^{-2j/3}$ is no greater than  $(\log_2 y)^{-\log_2 y/\log_3 y} = \frac{1}{\log y}$. Substituting for $(\log_2 y)^{-j}$ with $\frac{1}{\log y}$ and $j+1$ with $\frac{M}{\log_3 y}$, which is allowed because $\frac{M}{\log_3 y} + 1$ is the same size as $\frac{M}{\log_3 y}$, the summation in \eqref{eq:13} becomes, after simplification,
\begin{equation}\label{eq:17}
\sum_{m = \frac{n}{2}}^{M\log y-1} \sum_{j = \frac{\log_2 y}{\log_3 y}}^{\frac{M}{\log_3 y}}  \log_3 y\Bigg(\frac{1}{m!}\bigg[\qty(\frac{8M\log M}{\log y})^m  +\qty(\frac{8M\log_2 y}{\log y})^m\bigg]\Bigg).
\end{equation}

We remove the summation in \eqref{eq:17} by multiplying the summand by $\frac{M}{\log_3 y}$, truncate the resulting Taylor series, and apply Lemma \ref{lemma:2} to obtain the final contribution from this case as 
\begin{equation}\label{eq:18}
\frac{(8c)^{n/2}}{(n/2)!}\qty[\frac{(c\log_2 y)^{n/2+1}(\log_3 y)^{n/2}}{(\log y)^{n/2}} + \frac{(\log_2 y)^{n+1}}{(\log y)^{n/2}}] =  O_n\qty(\frac{(\log_2y)^{n+1}}{(\log y)^{n/2}}).
\end{equation}

\noindent\emph{Case 2:} $\ell$ is odd.
We proceed analogously to the even $\ell$ case, noting that if an arbitrary function $f$ is $O(h^{\ell/2-1/7\ell+\epsilon})$, then $f$ is also $ O(h^{\ell/2})$. Therefore, for odd $\ell$, \eqref{eq:42} is less than \begin{equation}\label{eq:8}\sum_{k = 0}^{M-1}
    \sum_{\substack{h = k\log y+1\\ h \equiv b-a \pmod{q}}}^{(k+1)\log y} g^h \sum_{\substack{\ell=n \\ \ell \text{ odd}}}^{h-1} (-z )^{\ell}O(h^{\ell/2}). 
\end{equation}
For $\ell \in [0, M-1]$, let $C_{\ell}$ be the implied constant in the $O(h^{\ell/2})$ term. Defining $C_{\max} = \max\{C_\ell\}$ allows us to pull $-C_{\max}$ out of the sum and remove the big $O$ notation. We also switch the order of sums in \eqref{eq:8} to obtain 
\begin{equation}\label{eq:9}
-C_{\max}\sum_{\substack{\ell = n \\ \ell \text{ odd}}}^{M\log y} \sum_{h > \max\{\ell, \log y\}}^{M\log y}g^h z^{\ell} h^{\ell/2}. 
\end{equation}

Since $h \ge \log y$, we know $g^h \le e^{-2h/3\log y}$. It thus follows that $z < \frac{2}{\log y}$ and $h \le M\log y$. Thus, maximizing $g^h$, $z^{\ell}$, and $h^{\ell/2}$ implies that \eqref{eq:9} has an upper bound of \begin{equation*}\label{eq:10}
-C_{\max}\sum_{\substack{\ell = n \\ \ell \text{ odd}}}^{M\log y} \qty(\frac{2}{\log y})^{\ell} (M\log y)^{\ell/2}\sum_{h > \max\{\ell, \log y\}}^{M\log y}e^{-2h/3\log y}.
\end{equation*} The sum over $h$ is a geometric series that is less than $\frac{e^{-2/3}}{1-e^{-2/3\log y}}$, which is less than $\log y$ for $\log y > 1$. Next, we distribute the $\qty(\frac{2}{\log y})^{\ell}$ into $(M\log y)^{\ell/2}$ and sum the resulting geometric series; this yields \begin{equation}\label{eq:11}
-C_{\max}\log y \frac{(4M/\log y)^{n/2}(1 - (4M/\log y)^{M\log y+1})}{1 - M/\log y}.
\end{equation}
For large $y$, both $1 - (M/\log y)^{M\log y+1}$ and $1 - M/\log y$ are $O(1)$. Thus, \eqref{eq:11} becomes \begin{equation}\label{eq:12}
-C_{\max} \frac{M^{n/2}}{(\log y)^{n/2-1}} = O_n\qty(\frac{(\log_2 y)^{n/2}}{(\log y)^{n/2-1}}).
\end{equation}

Note that for sufficiently large $y$, each of the cases based on $j$ are smaller than \eqref{eq:16}. Thus, the contributions from Cases 1A, 1B, 1D, and 2 as stated in \eqref{eq:34}, \eqref{eq:35}, \eqref{eq:18}, and \eqref{eq:12}, respectively, are all smaller than the contribution from Case 1C as stated in \eqref{eq:16}. Therefore,  $S_{\varnothing} = O_n\qty( \frac{(\log_2 y)^{n}}{(\log y)^{n/2-1}}),$ as desired.

Lemma \ref{lemma:4} implies that summations of $B_{h,\ell}$, $C_{h, \ell}$, or $D_{h, \ell}$ are closely related to summations of $A_{h, \ell}$. In order to take advantage of the cancellation suggested by the form $B_{h-1, \ell-1} = A_{h, \ell} - A_{h-1, \ell}$, we consider the sign of $A_{h, \ell}'' = \frac{\partial^2}{\partial h^2} A_{h, \ell}$. Namely, if $A_{h, \ell}'' > 0$, then \[A'_{h-1, \ell} < A_{h, \ell} - A_{h-1, \ell} < A'_{h, \ell}.\] Otherwise, if $A_{h, \ell}'' < 0$, then \[A'_{h-1, \ell} > A_{h, \ell} - A_{h-1, \ell} > A'_{h, \ell}.\]  Regardless of the sign of $A''_{h, \ell}$, we insert the appropriate upper bound given by either $A'_{h, \ell}$ or $A'_{h-1, \ell}$ into $S_{\{0\}}$ and $S_{\{h\}}$. In evaluating $S_{\{0, h\}}$, we take $A'''_{h, \ell}$ and use appropriate bounds for $A_{h, \ell} - 2A_{h-1, \ell} + A_{h-2, \ell}$.

We proceed to evaluate $S_{\{0\}}$, $S_{\{h\}}$, and $S_{\{0, h\}}$ in an analogous manner to the method of evaluating $S_{\varnothing}$. In loose terms, taking $k$ derivatives of $A_{h, \ell}$ corresponds to adding a factor of $(\log y)^k$ to the denominator of the asymptotic in Theorem \ref{thm:1}, thus leading to $S_{\{0\}}\log y$, $S_{\{h\}}\log y$, and $S_{\{0, h\}}(\log y)^2$. 

Recall that from the definition of $S_{\varnothing}$, $S_{\{0\}}$, $S_{\{h\}}$, and $S_{\{0, h\}}$, the relevant contribution to $\mathcal{D}_{\ge n}(a,b;y)$, after discarding terms where $h > M\log y$, is $S_{\varnothing} + S_{\{0\}} + S_{\{h\}} +S_{\{0, h\}}$. Therefore, $\mathcal{D}_{\ge n}(a, b; y)$ is $O_n\qty( \frac{(\log_2 y)^{n}}{(\log y)^{n/2-1}})$ as well. Since \[\mathcal{D}_n(a,b;y) = \mathcal{D}_{\ge n+1}(a,b;y) - \mathcal{D}_{\ge n}(a,b;y),\] it is also $O_n\qty( \frac{(\log_2 y)^{n}}{(\log y)^{n/2-1}})$. Thus, the theorem is proved. 
\end{proof}

\section{Proofs of the Lemmas}\label{app:a}

We now prove the lemmas that were used to prove the main theorem. 

\begin{lemma}\label{lemma:1}
Let $a$ and $b$ be nonnegative real numbers and $n$ be a positive integer. Then \[(a+b)^n \le 2^{n-1}(a^n+b^n).\] 
\end{lemma}
\begin{proof}
Since $x^n$ is convex for nonnegative $x$ and positive integers $n$, Jensen's inequality yields $(\frac{a}{2} + \frac{b}{2})^n \le \frac{1}{2}a^n + \frac{1}{2}b^n$. Clearing denominators gives $(a+b)^n \le 2^{n-1}(a^n+b^n)$, as desired. 
\end{proof}

\begin{lemma}\label{lemma:2}
For any real constants $a$ and $c$, we have \[\lim\limits_{x\to\infty} (\log x)^{c(\log_2 x)^a/\log x} = 1.\]
\end{lemma}
\begin{proof}
Since the limit \[L = \lim_{x\to\infty} (\log x)^{(\log_2 x)^a/\log x}\] does not depend on $c$, we set $c = 1$; if we prove $L = 1$ then certainly $L^c=1$ and the lemma follows. Taking logarithms, it suffices to show that \[\lim_{x\to\infty} \frac{(\log_2 x)^a}{\log x} = 0.\] However, any power of $\log_2 t$ grows slower than $\log t$ for all sufficiently large $t$, so the limit indeed equals $0$. The lemma is thus proved.
\end{proof}

\begin{lemma}\label{lemma:3}
Let $N$ be a real number and $\pr[g_n > x]$ denote the probability that the gap $g_n$ between the $n$th and $(n+1)$th prime is greater than $x$ for $1 \le n \le N$. Then \[\lim_{N \to\infty} \pr[g_n>c\log_2 p_N\log p_n] < \frac{1}{(\log N)^{c}}.\]
\end{lemma}
We sketch the details of the proof here. Although not fully rigorous, we expect the key ingredients of the proof to be present. 
\begin{proof}
In the following proof, we omit the limits as $N$ goes to infinity for readability. 
Gallagher \cite{gallagher1976distribution} showed that
\[\pr[1 \le n \le N \mid g_n > \lambda\log p_n] < e^{-\lambda}.\] Setting $\lambda = c\log_2 p_N$ implies $\pr[g_n > c\log_2 p_N\log p_n] < \frac{1}{(\log p_N)^c}$. It is clear that $p_N > N$; hence \[\frac{1}{(\log p_N)^c} < \frac{1}{(\log N)^c}.\] Thus, 
\[\pr[g_n > c\log_2 p_N \log p_n] < \frac{1}{(\log N)^{c}},\]
as desired.
\end{proof}

\begin{lemma}\label{lemma:4}
The sums over subsets of $[1, h-1]$ of size $\ell$, given by $A_{h,\ell}$, $B_{h, \ell}$, $C_{h, \ell}$, and $D_{h, \ell}$, satisfy the following relations:
\begin{align*}
    &B_{h-1, \ell-1} = C_{h-1, \ell-1} = A_{h, \ell} - A_{h-1, \ell},\\
    &D_{h-1, \ell-1} = A_{h, \ell} - 2A_{h-1, \ell} + A_{h-2, \ell}.
\end{align*}
\end{lemma}
\begin{proof}

Note that $A_{h,\ell}$ is a sum that ranges over all subsets $\tset$ of $[1, h-1]$. We can partition this sum by $\max\{\tset\}$. Setting $m = \max\{\tset\}$, we can write 
\begin{equation}\label{eq:31}
A_{h, \ell} = \sum_{m=\ell}^{h-1} \sum_{\substack{\tset \in [1, m-1] \\ \abs{T} = \ell-1 \\ (t+a, q) = 1}} \mathfrak{S}_{q,0}(\{m\} \cup \tset).
\end{equation} From Definition \ref{def:1},  $\mathfrak{S}_{q,0}(\tset) = \mathfrak{S}_{q,0}(s - \tset)$ for any integer $s$.  Using the translational invariance of $\mathfrak{S}_{q,0}$ and noting that \[C_{m, \ell-1} = \sum\limits_{\substack{\tset \in [1, m-1] \\ \abs{T} = \ell-1 \\ (t+a, q) = 1}} \mathfrak{S}_{q,0}(\{m\} \cup \tset),\] we can rewrite \eqref{eq:31} as $A_{h, \ell} = \sum\limits_{m=\ell}^{h-1} C_{m, \ell-1}$. Now consider $A_{h, \ell} - A_{h-1, \ell}$; every term in this difference cancels except $ C_{h-1, \ell-1}$, so $A_{h, \ell} - A_{h-1, \ell} = C_{h-1, \ell-1}$, as desired. 

Similarly, we can partition a sum over subsets $\tset$ of $[1,h-1]$ to a sum over sets $\tset$ whose minimum value is $m$. Thus \[A_{h, \ell} = \sum_{m=\ell}^{h-1} \sum_{\substack{\tset \in [m-\ell+1, h-1] \\ \abs{T} = \ell-1 \\ (t+a, q) = 1}} \mathfrak{S}_{q,0}(\{m-\ell\} \cup \tset).\] Translational invariance implies that \[B_{m, \ell-1} = \sum_{\substack{\tset \in [m-\ell+1, h-1] \\ \abs{T} = \ell-1 \\ (t+a, q) = 1}} \mathfrak{S}_{q,0}(\{m-\ell\} \cup \tset),\] so $A_{h, \ell} - A_{h-1, \ell}$ telescopes as before and only $B_{h-1, \ell-1}$ remains.
Therefore, $A_{h, \ell} - A_{h-1, \ell} = B_{h-1, \ell-1}$ as well. 

Finally, in order to relate $A_{h, \ell}$ to $D_{h, \ell}$, we write the sum over subsets of $[1, h-1]$ as a sum over $m$ and over sets $\tset$ where $\max\{\tset\} - \min\{\tset\} = m$. By translational invariance, because there are $h-m+1$ possibilities for $\min\{\tset\}$, there are $h-m+1$ copies of $\mathfrak{S}_{q,0}(\{1, m\} \cup \tset)$. Hence, the definition for $A_{h, \ell}$ can be rewritten as
\begin{equation}\label{eq:32}
A_{h, \ell} = \sum_{m=\ell}^{h-1} \sum_{\substack{\tset \in [2, m-1] \\ \abs{T} = \ell-2 \\ (t+a, q) = 1}} (h-m+1)\mathfrak{S}_{q,0}(\{1, m\} \cup \tset). 
\end{equation}  Recall that $D_{h,\ell} =  \displaystyle\sum_{\substack{\tset \subset [1, h-1] \\ (t+a, q)=1\\ \abs{\tset} = \ell} } \mathfrak{S}_{q,0}(\{0,h\} \cup \tset)$, so $A_{h, \ell} - A_{h-1, \ell} = \sum\limits_{m=\ell}^{h-1} D_{m, \ell-1}$. Therefore, we express \eqref{eq:32} as 
$A_{h, \ell} = \sum\limits_{m = \ell}^{h-1} (h-m+1)D_{m, \ell-2}$.  
Note that the sum telescopes when two successive differences are taken. What remains is $D_{h-1, \ell-2} = (A_{h, \ell} - A_{h-1, \ell}) - (A_{h-1, \ell} - A_{h-2, \ell}) = A_{h, \ell} - 2A_{h-1, \ell} + A_{h-2, \ell}$, as desired.
\end{proof}

\section{Acknowledgments} 
The author would like to thank his mentor Robert Burklund for his extremely helpful guidance. He is also very grateful for the extensive advice and feedback provided by John Rickert and Tanya Khovanova. The author would like to thank Robert Lemke Oliver for his helpful comments on my ideas. Finally, the author would like to thank Lawrence Washington for sacrificing his own time to explain the background of the project.

%

\appendix
\section{Numerical Results}\label{sec:numeric}

The following  simplified asymptotic for the case $\pi(x; 3, (a,b))$ is provided in \cite{oliver2016}: \begin{equation}\label{eq:19}
\pi(x; 3, (a, b)) = \frac{\li(x)}{4}\qty(1 \pm \frac{1}{2\log x}\log\qty(\frac{2\pi\log x}{q})) + O\qty(\frac{x}{(\log x)^{11/4}}),
\end{equation} 
with the plus or minus sign being plus if $a \neq b$ and minus if $a = b$.

We compare \eqref{eq:19} to the actual behavior of the primes, and find that because the approximations that were necessary to arrive at \eqref{eq:19}, the data deviate from the conjectured form. Using SageMath's \texttt{find\_fit} function suggested a possible lower order term of size $O\qty(\frac{(\log_2 x)^{2}}{(\log x)^2})$. We modified the data gathering process to approximate values of $\pi(x; q, (a,b))$ for $x \le 10^{18}$. Finally, we include more terms in the approximation of $\mathcal{D}(a, b; y)$ to improve its accuracy in future work. Graphs may be found in Appendix \ref{app:c}.

Lemke Oliver and Soundararajan \cite{oliver2016} gathered values of $\pi(x; q, (a, b))$ up to $x = 10^{12}$. We gathered data up to $x = 10^{18}$. We gathered complete raw data using SageMath for $1 \le x \le 10^{10}$. For $10^{10} < x \le 10^{18}$, a sampling technique was used to approximate the ratio $\pi(x; q, (a,b))/\pi(x)$. Lemke Oliver's \texttt{C++} code counts prime patterns in fixed intervals $[X, Y)$; the program was modified to only consider the first $10^8$ primes larger than $X$. We used $X = 10^{b_1}$ and $Y = 10^{b_1+1}$ for $10 \le b_1 \le 18$. The program estimated pattern frequencies at $X + i \cdot 10^{b_1}$ for $1 \le i \le 9$. 

Theorem \ref{thm:1} shows that the contributions $S_{\varnothing}$, $S_{\{0\}}$, $S_{\{h\}}$, and $S_{\{0, h\}}$ to $\mathcal{D}(a, b; y)$ decline quickly with $n$. After dividing by $\li(x)$, the main terms in the main conjecture in \cite{oliver2016} are of size $O(1)$, $O(\frac{\log \log x}{\log x})$, and $O(\frac{1}{\log x})$. When $\abs{\tset} = 6$, Theorem \ref{thm:1} implies that $S_{\varnothing} \in O\qty(\frac{(\log_2 y)^6}{(\log y)^2})$. Thus, for $n \ge 6$, $S_{\varnothing}$, $S_{\{0\}}$, $S_{\{h\}}$, and $S_{\{0, h\}}$ make negligible contributions to $\mathcal{D}$. This implies that long range correlations between prime patterns are negligible, which in turn implies that even though we only take the first $10^8$ primes after $X$, the sample can be reasonably assumed to be unbiased. 

Following Cram\'{e}r's model, we model primality as a binomial event with $x$ being prime with probability $\frac{1}{\log x}$ and assume that primality of $x$ and $y$ are independent events. Then the standard deviation of our sampling distribution is proportional to $\frac{1}{\sqrt{C}}$, where $C$ is the number of primes sampled in order to estimate the frequency of $\pi(x; q, (a, b))$ at $x$. 

For each point estimate at $X + i \cdot 10^{b_1}$, we sampled with $C = 10^8$, giving a precision of roughly $10^{-4}$. The sample gives a sampling frequency \[f_{a, b} = \frac{\pi(x+x_0; q, (a, b)) - \pi(x; q, (a,b))}{\pi(x+x_0) - \pi(x)},\] where $\pi(x+x_0) - \pi(x) = 10^7$. In a crude sense, the sampling frequency $f_{a,b}$ is the derivative of $\pi(x; q, (a,b))$, so we used a Riemann sum with $10$ equally spaced subintervals to estimate $\pi(X+i \cdot 10^{b_1};q,(a,b))$ from $f_{a,b}$. We thus computed \begin{equation}\label{eq:36}
\frac{\sum_{\beta = 1}^{b_1} \sum_{\alpha = 1}^9 f_{a, b}[\li((\alpha+1) \cdot 10^{\beta}) - \li(\alpha \cdot 10^{\beta})]}{\li(9 \cdot 10^{b_1})}.
\end{equation} Note that \eqref{eq:36} approximates the the ratio $\frac{\pi(10^{b_1+1}; q, (a, b))}{\pi(x)}$ and hence allows us to extend our data to $x = 10^{18}$. 

We restate the conjectured form for $\pi(x; q, (a,b))$ as in Conjecture \ref{conjec:1} for convenience as \begin{equation}\label{conjec:eq}
\pi(x; q, (a, b)) \sim \frac{1}{q}\int_{2}^{x} \alpha(y)^{\epsilon_q(a,b)}\qty(\frac{q}{\varphi(q)\alpha(y)\log y})^2 \mathcal{D}(a,b;y) dy.
\end{equation} The numerical model in \cite{oliver2016} is evaluated by partitioning $\mathcal{D}(a,b;y)$ into $\sum_n \mathcal{D}_n(a,b;y)$ and discarding $\mathcal{D}_{n}(a, b; y)$ for $n \ge 3$. Thus $\mathfrak{S}_{q,0}$ is estimated only for zero and two term sets to approximate $\mathcal{D}(a,b;y)$. For example, in \cite{oliver2016}, only the zero and two term sets for $\mathcal{D}_1(a,b;y)$ are considered. Lemke Oliver and Soundararajan then write \begin{equation}\label{eq:50}\mathcal{D}_1(a,b;y) \approx -\frac{q}{\varphi(q)\alpha(y)\log y}\sum_{\substack{h > 0 \\ h \equiv b-a \pmod{q}}} \sum_{\substack{t \in [1, h-1] \\ (t+a, q) = 1}} \mathfrak{S}_{q,0}(\{0, t\}) + \mathfrak{S}_{q,0}(\{t, h\}). \end{equation}

However, Theorem \ref{thm:1} suggests that only considering zero and two term sets may not accurate enough. Hence, we add terms to $\mathcal{D}_0$, $\mathcal{D}_1$, and $\mathcal{D}_2$, as well as truncating at $\mathcal{D}_5$ instead of $\mathcal{D}_2$ to approximate $\mathcal{D}$ in \eqref{conjec:eq}. For example, recalling that $\mathcal{D}_1(a,b;y)$ contains all terms of \eqref{eq:1} with $\abs{\tset} = 1$, we write \[\mathcal{D}_1(a,b;y) = -\frac{q}{\varphi(q)\alpha(y)\log y}\sum_{\substack{h > 0 \\ h \equiv b-a \pmod{q}}} \sum_{\substack{t \in [1, h-1] \\ (t+a, q) = 1}} \mathfrak{S}_{q,0}(\{0, t\}) + \mathfrak{S}_{q,0}(\{t, h\}) + \mathfrak{S}_{q,0}(\{0, t, h\}). \] This is essentially \eqref{eq:50} but with three term sets included. The values of singular series $\mathfrak{S}_{q, 0}(\mathcal{H})$ were computed up to five term sets with $\max \mathcal{H} \le 150$ and prepared for future work numerically integrating \eqref{conjec:eq} by truncating at $\mathcal{D}_5(a,b;y)$.  
\section{Plots of Data and Model}\label{app:c}

This appendix contains the plots of raw data, extended data, curve fitting. 
\begin{figure}[!htb]
    \centering
    \includegraphics[width=0.6\linewidth]{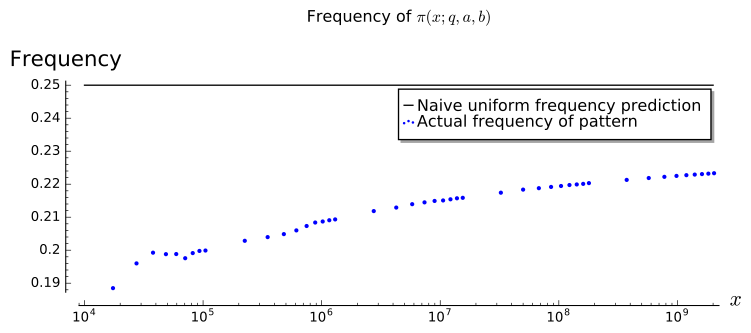}
    \caption{The proportion $\frac{\pi(x; 3, (1, 1))}{\pi(x)}$ for $x_0 \le x \le x_1$ where $\pi(x_0) = 10^4$ and $\pi(x_1) = 10^9$.}
    \label{fig:my_label1}
\end{figure}

\begin{figure}[!h]
\centering
\begin{minipage}{0.48\textwidth}
  \centering
  \includegraphics[width=\linewidth]{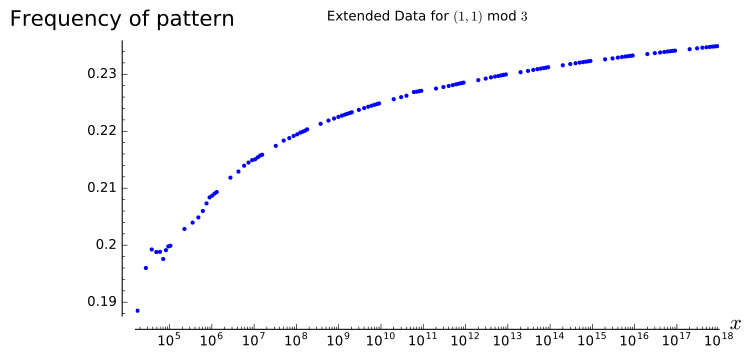}
  \captionof{figure}{The extended data for $(1, 1)$ modulo $3$. The slight bump at $5 \cdot 10^{10}$ is due to combining the raw and sampled data.}
  \label{fig:test1}
\end{minipage}\hfill
\begin{minipage}{0.48\textwidth}
  \centering
  \includegraphics[width=\linewidth]{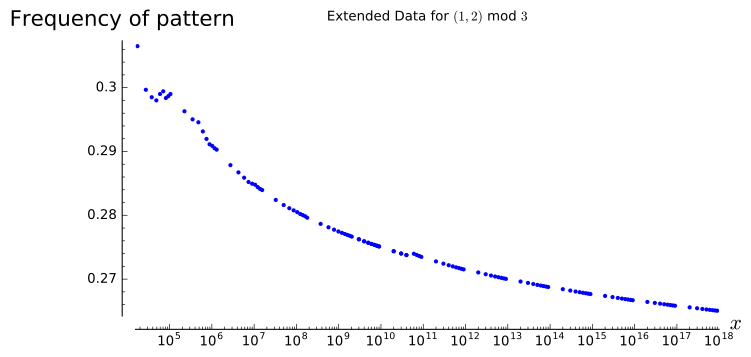}
  \captionof{figure}{The extended data for $(1, 2)$ modulo $3$. The slight bump at $5 \cdot 10^{10}$ is due to stitching the raw data and sampled data together.}
  \label{fig:test2}
\end{minipage}
\end{figure}

\begin{figure}[!h]
\centering
\begin{minipage}{.48\textwidth}
  \centering
  \includegraphics[width=\linewidth]{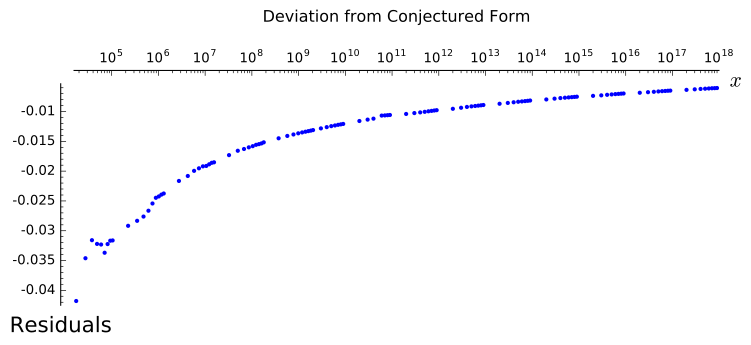}
  \captionof{figure}{The residuals when \eqref{eq:19} is subtracted from $\pi(x; 3, (1, 1))$ for $x_0\le  x \le x_1$ where $\pi(x_0) = 10^4$ and $\pi(x_1) = 10^{18}$, using the extended data.}
  \label{fig:test3}
\end{minipage}\hfill
\begin{minipage}{.48\textwidth}
  \centering
  \includegraphics[width=\linewidth]{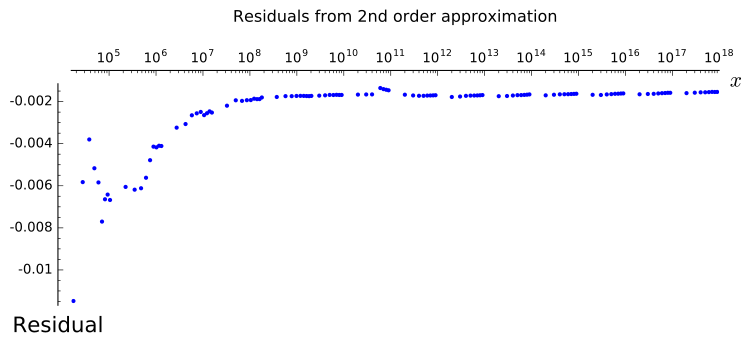}
  \captionof{figure}{The residuals when curve fitted terms of size $O((\log_2 y)^2/(\log y)^2)$ and \eqref{eq:19} are subtracted from $\pi(x; 3, (1, 1))$ for $x_0\le  x \le x_1$ where $\pi(x_0) = 10^4$ and $\pi(x_1) = 10^{18}$, using the extended data.}
  \label{fig:test4}
\end{minipage}
\end{figure}
\eject

\bibliography{bib}
\bibliographystyle{ieeetr}

\end{document}